\documentclass[12pt]{article}
\usepackage[misc]{ifsym}
\usepackage[T1]{fontenc}
\usepackage[utf8]{inputenc}
\usepackage{authblk}
\usepackage{indentfirst}
\usepackage[top=1in, bottom=1in, left=1in, right=1in]{geometry}

\usepackage{amsmath}
\usepackage{enumerate}
\usepackage{amsthm}
\usepackage{abstract}

\usepackage{bm}
\usepackage{graphics}
\usepackage{graphicx}
\usepackage{subfigure}
\usepackage{caption}
\usepackage{amsthm,amsmath,amssymb}
\usepackage{float}
\usepackage[section]{placeins}
\usepackage{multirow}
\usepackage{amssymb}
\usepackage{IEEEtrantools}
\usepackage{cite}
\usepackage{pifont}
\usepackage{mathrsfs}
\pdfoutput=1

\title{Brualdi–Hoffman–Tur\'{a}n problem of the gem}
\date{}
\author{Fan Chen$^{1,\, 2}$, Xiying Yuan$^{1,\, 2}$ \thanks{Corresponding author.\\ \indent This work is supported by the National Natural Science Foundation of China (Nos. 11871040,
12271337, 12371347) \\ \indent Email address: xiyingyuan@shu.edu.cn (Xiying Yuan) chenfan@shu.edu.cn (Fan Chen)}}
\affil{{\footnotesize\emph{1. Department of Mathematics, Shanghai University, Shanghai 200444, P.R. China}}\\
{\footnotesize\emph{2. Newtouch Center for Mathematics of Shanghai University, Shanghai 200444, P.R. China}}}

\begin{document}
\newtheorem{theorem}{Theorem}[section]
\newtheorem{assumption}[theorem]{Assumptio}
\newtheorem{corollary}[theorem]{Corollary}
\newtheorem{proposition}[theorem]{Proposition}
\newtheorem{lemma}[theorem]{Lemma}
\newtheorem{definition}[theorem]{Definition}
\newtheorem{remark}[theorem]{Remark}
\newtheorem{problem}[theorem]{Problem}
\newtheorem{claim}{Claim}
\newtheorem{conjecture}[theorem]{Conjecture}
\newtheorem{fact}{Fact}

\maketitle
\noindent\rule[0pt]{17.5cm}{0.09em}

\noindent{\bf Abstract}

A graph is said to be $F$-free if it does not contain $F$ as a subgraph. Brualdi–Hoffman– Tur\'{a}n type problem seeks to determine the maximum spectral radius of an $F$-free graph with given size. The gem consists of a path on $4$ vertices, along with an additional vertex that is adjacent to every vertex of the path. Concerning Brualdi–Hoffman–Tur\'{a}n type problem of the gem, when the size is odd, Zhang and Wang [Discrete Math. 347 (2024) 114171] and Yu, Li and Peng [arXiv:2404. 03423] solved it. In this paper, we completely solve the Brualdi–Hoffman–Tur\'{a}n type problem of the gem.

\noindent

\noindent{\bf Keywords:} Gem; Spectral radius; Brualdi–Hoffman–Tur\'{a}n type problem.

\noindent{\bf AMS subject classifications:} 05C50; 05C35

\noindent\rule[0pt]{17.5cm}{0.05em}

\section{Introduction}
Let $G$ be a simple graph with vertex set $V(G)$ and edge set $E(G)$. We denote the number of vertices and edges of  $G$ by $n=|V(G)|$ and $m=|E(G)|$, respectively. We use the notation  $P_n$, $C_n$, and $K_n$ to represent the path, cycle, and complete graph on $n$ vertices, respectively. For two disjoint graphs $G$ and $H$, let $G\cup H$ be the graph with vertex set $V(G)\cup V(H)$ and edge set $E(G)\cup E(H)$. The join of two disjoint graphs $G$ and $H$, denoted by $G\vee H$, is obtained from $G\cup H$ by adding all possible edges between $G$ and $H$. For a graph $G$, let $A(G)$ denote its adjacency matrix and $\rho(G)$ denote the spectral radius of $G$. By Perron-Frobenius Theorem, $\rho(G)$ is the largest eigenvalue of $A(G)$.

In \cite{N1}, Nikiforov introduced the classic spectral Tur\'{a}n problem, which aims to determine the maximum spectral radius of an $F$-free graph of given order. This problem is also known as the Brualdi-Solheid-Tur\'{a}n problem. The study of the Brualdi-Solheid-Tur\'{a}n type problem has a rich history, such as \cite{N2, W, BG, N3, N4, N5, CDT} and so on. For a comprehensive overview of the developments in this area, we refer the reader to the survey in \cite{LN}.

It is a natural question to ask what is the maximum spectral radius of an $H$-free graph with given size, which is commonly referred to as Brualdi-Hoffman-Tur\'{a}n type problem (see \cite{BH}). Li, Zhai, and Shu \cite{LZS} proposed a conjecture on cycle-version.
\begin{conjecture}\cite{LZS}\label{conjec1.2}
For $k\geqslant 2$ and $m$ sufficiently large, if $G$ is a $C_{2k+1}$-free or $C_{2k+2}$-free with given size $m$, we have 
$\rho(G)\leqslant \frac{k-1+\sqrt{4m-k^2+1}}{2}$, the equality holds if and only if $G\cong K_k\vee(\frac{m}{k}-\frac{k-1}{2})K_1$.
\end{conjecture}
Zhai, Lin and Shu \cite{ZLS} solved Conjecture \ref{conjec1.2} for $k=2$, while the extremal graph $K_2\vee\frac{m-1}{2}K_1$ exits only for odd size $m$. Then Min, Lou and Huang \cite{MLH} and Sun, Li and Wei \cite{SLW} characterized the extremal graph for even size $m$ when $k=2$. Let $C_k^+$ denote the graph obtained from the $C_k$ by adding an edge between two vertices of distance two. Observe that if $G$ is $C_k$-free, then $G$ is $C_k^+$-free. A strengthening version of Conjecture \ref{conjec1.2} was proposed in \cite{LZS}.
\begin{conjecture}\cite{LZS}\label{conjec1.3}
For $k\geqslant 2$ and $m$ sufficiently large, if $G$ is a $C_{2k+1}^+$-free or $C_{2k+2}^+$-free with given size $m$, we have 
$\rho(G)\leqslant \frac{k-1+\sqrt{4m-k^2+1}}{2}$, the equality holds if and only if $G\cong K_k\vee(\frac{m}{k}-\frac{k-1}{2})K_1$.
\end{conjecture}

Conjecture \ref{conjec1.3} has been solved for $C_5^+$ and $C_6^+$ \cite{SLW}, namely when $m$ is odd, Brualdi-Hoffman-Tur\'{a}n type problem of $C_5^+$ and $C_6^+$ have been solved. When $m$ is even, Fang and You \cite{FY} and Liu and Wang \cite{LW} solved Brualdi-Hoffman-Tur\'{a}n type problem of $C_5^+$ and $C_6^+$ respectively. Recently, Conjecture \ref{conjec1.3} has been completely confirmed by Li, Zhai, and Shu in \cite{LZS}. Let $H_t=K_1\vee P_{t-1}$. Observe that if $G$ is $C_k^+$-free, then $G$ is $H_k$-free. Yu, Li and Peng \cite{YLP} proposed a strengthening version of Conjecture \ref{conjec1.3}.
\begin{conjecture}\cite{YLP}\label{conjec1.4}
For $k\geqslant 2$ and $m$ sufficiently large, if $G$ is an $H_{2k+1}$-free or $H_{2k+2}$-free with given size $m$, we have 
$\rho(G)\leqslant \frac{k-1+\sqrt{4m-k^2+1}}{2}$, the equality holds if and only if $G\cong K_k\vee(\frac{m}{k}-\frac{k-1}{2})K_1$.
\end{conjecture}
Recently, Li, Zhou, and Zou \cite{LZZ} completely solved Conjecture \ref{conjec1.4}. When $t=5$, $H_5$ is also called the gem. Zhang and Wang \cite{ZW} and Yu, Li, Peng \cite{YLP} solved Conjecture \ref{conjec1.4} for the case $H_5$. 
\begin{theorem}\cite{ZW, YLP}\label{gem}
If $G$ is a gem-free graph of size $m\geqslant 11$ without isolated vertices, then
$\rho(G)\leqslant \frac{1+\sqrt{4m-3}}{2},$
the equality holds if and only if $G\cong K_2\vee\frac{m-1}{2}K_1$.
\end{theorem}

Let $S_{n,k}$ be the graph obtained by joining each vertex of $K_k$ to $n-k$ isolated vertices. Note that $K_2\vee\frac{m-1}{2}K_1\cong S_{\frac{m+3}{2},2}$. Let $S_{n,2}^{-t}$ be the graph obtained from $S_{n,2}$ by deleting $t$ edges between a vertex of degree $n-1$ and $t$ vertices of degree $2$. Let $\rho^*(m)$ be the largest root of $x^4-mx^2-(m-2)x+(\frac{1}{2}m-1)=0$. In \cite{MLH}, it was pointed out that $\rho(S_{\frac{m+4}{2},2}^{-1})=\rho^*(m)$. 

The extremal graph of the gem in Theorem \ref{gem} exits only when $m$ is odd. Motivated by this observation, in this paper, we completely solve the Brualdi-Hoffman-Tur\'{a}n type problem of the gem by proving the following Theorem \ref{th1}. 

\begin{theorem}\label{th1}
If $G$ is a gem-free graph of size $m\geqslant92$ without isolated vertices, then 
$$\rho(G)\leqslant \left\{
                  \begin{array}{ll}
                    \frac{1+\sqrt{4m-3}}{2}, & \hbox{if $m$ is odd,} \\
                   \rho^*(m), & \hbox{if $m$ is even.}
                  \end{array}
                \right.
$$
Moreover, the equality holds if and only if $G\cong \left\{
                  \begin{array}{ll}
                    S_{\frac{m+3}{2},2}, & \hbox{if $m$ is odd,} \\
                    S_{\frac{m+4}{2},2}^{-1}, & \hbox{if $m$ is even.}
                  \end{array}
                \right.$
\end{theorem}

\section{Preliminaries}
For a vertex $u\in V(G)$, let $N_G(u)$ be the neighborhood of $u$. For convenience, we use $N_S(u)$ to refer to the intersection $N_G(u)\cap S$, where $S\subseteq V(G)$. For a subset $A\subseteq V(G)$, we denote by $e(A)$ the number of edges with both endpoints in $A$. For two disjoint sets $A$, $B$, we define $e(A, B)$ to be the number of edges between $A$ and $B$. 

\begin{lemma}\cite{WXY}\label{big}
Let $u$, $v$ be two distinct vertices in a connected graph $G$. Suppose that $v_1, v_2, \cdots, v_s (1\leqslant s\leqslant d_G(v))$ are some vertices of $N_G(v)\setminus N_G(u)$, and $\bm{x}$ is a Perron vector of $A(G)$ with coordinate $x_z$ corresponding to $z\in V(G)$. Let $G'=G-\{vv_i|1\leqslant i\leqslant s\}+\{uv_i|1\leqslant i\leqslant s\}$. If $x_u\geqslant x_v$, then $\rho(G')>\rho(G)$.
\end{lemma}

\begin{lemma}\cite{ZLS}\label{con}
Let $G$ be an extremal $F$-free graph without isolated vertices with given size $m$. If $F$ is a 2-connected graph, then $G$ is connected.
\end{lemma}

Let $G$ be a graph achieving maximum spectral radius of gem-free graphs without isolated vertices with $m\geqslant 92$ edges. By Lemma \ref{con}, $G$ is connected. Suppose $\rho$ is the spectral radius of $G$ and $\bm{x}$ is the positive eigenvector with $x_{u^*}=\max\{x_v:v\in V(G)\}=1$. Denote by $U=N_G(u^*)$, $W=V(G)\setminus (N_G(u^*)\cup \{u^*\})$ and $d_U(u) = |N_{G[U]}(u)|$ for a vertex $u\in V(G)$. Since
\begin{equation}\label{e2}
\rho^2=\rho^2x_{u^*}=d_G(u^*)+\sum_{u\in U}d_U(u)x_u+\sum_{w\in W}d_U(w)x_w,
\end{equation}
which together with $\rho=\rho x_{u^*}=\sum_{u\in U}x_u$ yields
\begin{equation}\label{e3}
\rho^2-\rho=d_G(u^*)+\sum_{u\in U}(d_U(u)-1)x_u+\sum_{w\in W}d_U(w)x_w.
\end{equation}
For an arbitrary vertices set $V$, define
\begin{equation}\label{def}
\eta_1(V)=\sum_{u\in V}(d_V(u)-1)x_u-e(V).
\end{equation}
If $V=\emptyset$, then set $\eta_1(V)=0$.

\begin{lemma}\cite{FY}\label{t-}
For $m>t+1$, $\rho(S_{\frac{m+t+3}{2},2}^{-t})<\left\{
                  \begin{array}{ll}
                   \rho(S_{\frac{m+3}{2},2}), & \hbox{if $t(\geqslant 2)$ is even,}\\
                    \rho(S_{\frac{m+4}{2},2}^{-1}), & \hbox{if $t(\geqslant 3)$ is odd.} 
                  \end{array}
                \right.$ 
\end{lemma}

To prove the main result Theorem \ref{th1}, we will prove that $G[U]\cong K_{1,a}\cup bK_1$ with $a\geqslant 1$ (see Lemma \ref{star}) and $W=\emptyset$ (see Lemma \ref{W}). Then we have $G\cong S_{\frac{m+t+3}{2},2}^{-t}$. Note that $e(S_{\frac{m+3}{2},2})=e(S_{\frac{m+t+3}{2},2}^{-t})=m$ for $t\geqslant 1$. If $m$ is odd, then $t$ is even, and Lemma \ref{t-} shows that $\rho(S_{\frac{m+3}{2},2})>\rho(S_{\frac{m+t+3}{2},2}^{-t})$ for $t\geqslant 2$. Hence $G\cong S_{\frac{m+3}{2},2}$. If $m$ is even, then $t$ is odd, and Lemma \ref{t-} shows that $\rho(S_{\frac{m+4}{2},2}^{-1})>\rho(S_{\frac{m+t+3}{2},2}^{-t})$ for $t\geqslant 3$. Hence $G\cong S_{\frac{m+4}{2},2}^{-1}$.

\section{Characterization of the extremal graph of the gem}
Recall that $\rho$ is the spectral radius of an extremal graph $G$ and $\bm{x}$ is the positive eigenvector with $x_{u^*}=\max\{x_v:v\in V(G)\}=1$, $U=N_G(u^*)$, $W=V(G)\setminus (N_G(u^*)\cup \{u^*\})$. In this section, we mainly prove that $G[U]\cong K_{1,a}\cup bK_1$ with $a\geqslant 1$ (see Lemma \ref{star}) and $W=\emptyset$ (see Lemma \ref{W}). 

Noting that $\rho\geqslant\rho^*(m)>\frac{1+\sqrt{4m-5}}{2}$, we have
\begin{equation}\label{e1}
\rho^2-\rho> m-\frac{3}{2}.
\end{equation}
Then we have
\begin{equation}\label{23}
\rho^2-\rho> m-\frac{3}{2}=d_G(u^*)+e(U)+e(W)+e(U,W)-\frac{3}{2}.
\end{equation}
Combining (\ref{e3}) and (\ref{23}), we have
\begin{align}
&d_G(u^*)+e(U)+e(W)+e(U,W)-\frac{3}{2}\notag\\
&<\rho^2-\rho\notag\\
&=d_G(u^*)+\sum_{u\in U}(d_U(u)-1)x_u+\sum_{w\in W}d_U(w)x_w\notag\\
&\leqslant\ d_G(u^*)+\sum_{u\in U}(d_U(u)-1)x_u+e(U,W),
\end{align}
which implies
$$\sum_{u\in U}(d_U(u)-1)x_u-e(U)\geqslant e(W)-\frac{3}{2},$$
namely,
\begin{equation}\label{e5}
\eta_1(U)=\sum_{u\in U}(d_U(u)-1)x_u-e(U)\geqslant e(W)-\frac{3}{2}\geqslant -\frac{3}{2},
\end{equation}
and
\begin{equation}\label{e6}
e(W)\leqslant \eta_1(U)+\frac{3}{2}.
\end{equation}

First we will characterize the structure of $G[U]$. Recall that $G$ is gem-free, and then we have $G[U]$ is $P_4$-free. As a consequence, the diameter of the connected component of $G[U]$ is at most $2$. The connected component of $G[U]$ is $K_3$ or $K_{1, a}$ with $a\geqslant 0$. Let $\mathcal{H}$ be the family of connected components in $G[U]$ and $|\mathcal{H}|$ be the number of members in $\mathcal{H}$, and let $\mathcal{H}_1=\{H\in \mathcal{H}: H\cong K_3\}$, $\mathcal{H}_2=\{H\in \mathcal{H}: H\cong K_{1,a}\ \text{with}\  a\geqslant 1\}$, and $\mathcal{H}_3=\mathcal{H}\setminus(\mathcal{H}_1\cup \mathcal{H}_2)$. Hence, for $H\in \mathcal{H}_3$, $H$ is an isolated vertex. 

\begin{lemma}\label{star}
$G[U]\cong K_{1,a}\cup bK_1$ with $a\geqslant 1$.
\end{lemma}
\begin{proof}

For each connected component $H$ of $G[U]$, we claim that
\begin{equation}\label{h1}
\eta_1(V(H))\leqslant \left\{
                  \begin{array}{lll}
                    0, & \hbox{if $H\in \mathcal{H}_1$,} \\
                    -1, & \hbox{if $H\in \mathcal{H}_2$,}\\
                    0,& \hbox{if $H\in \mathcal{H}_3$.}
                  \end{array}
                \right.
\end{equation}
If $H\in \mathcal{H}_1$, then $H\cong K_3$. Hence
\begin{align*}
\eta_1(V(H))=&\sum_{u\in V(H)}(d_H(u)-1)x_u-e(H)\\
\leqslant&\sum_{u\in V(H)}(d_H(u)-1)-e(H)\\
=&e(H)-V(H)=0.
\end{align*}
If $H\in \mathcal{H}_2$, then $H\cong K_{1,a}$ with $a\geqslant 1$. Hence
\begin{align*}
\eta_1(V(H))=&\sum_{u\in V(H)}(d_H(u)-1)x_u-e(H)\\
\leqslant&\sum_{u\in V(H)}(d_H(u)-1)-e(H)\\
=&e(H)-V(H)=-1.
\end{align*}
If $H\in \mathcal{H}_3$, then $H$ is an isolated vertex. Hence
$$\eta_1(V(H))=\sum_{u\in V(H)}(d_H(u)-1)x_u-e(H)=-\sum_{u\in V(H)}x_u< 0.$$

Combining (\ref{e6}) and  (\ref{h1}), we have $e(W)\leqslant \eta_1(U)+\frac{3}{2}\leqslant \frac{3}{2}$. Hence, $e(W)\leqslant 1$.

\begin{claim}\label{H3}
$\mathcal{H}_1=\emptyset$.
\end{claim}
\begin{proof}[\rm{\textbf{Proof of Claim \ref{H3}}}]
Suppose to the contrary that $\mathcal{H}_1\neq\emptyset$. Then there is a component $H\cong K_3$ in $G[U]$, say $V(H)=\{v_1,v_2,v_3\}$. Suppose that $x_{v_1}\geqslant x_{v_2}\geqslant x_{v_3}$. 

We will claim that $N_W(v_2)=N_W(v_3)=\emptyset$. If $\cup_{i=1}^3N_W(v_i)=\emptyset$, then it is obvious that $N_W(v_2)=N_W(v_3)=\emptyset$. Suppose that $\cup_{i=1}^3N_W(v_i)\neq\emptyset$. For any $w\in \cup_{i=1}^3N_W(v_i)$, since $G$ is gem-free, we have $|N_G(w)\cap \{v_1,v_2,v_3\}|\leqslant 1$. Otherwise suppose $w\sim v_1$ and $w\sim v_2$. Hence $G[\{v_1,u^*,v_3,v_2,w\}]$ is the gem, which is a contradiction. We will prove that $N_G(w)\cap \{v_1,v_2,v_3\}=\{v_1\}$. Without loss of generality, suppose to the contrary that $N_G(w)\cap \{v_1,v_2,v_3\}=\{v_3\}$. Let $G'=G-wv_3+wv_1$. Recall that $e(W)\leqslant 1$. Then $G'$ is still a gem-free graph with size $m$. Since $x_{v_1}\geqslant x_{v_3}$, by Lemma \ref{big}, we have $\rho(G')>\rho(G)$, which is a contradiction. Thus we can get $N_W(v_2)=N_W(v_3)=\emptyset$. Hence, by symmetry, we have $x_{v_2}=x_{v_3}$ and $\rho x_{v_3}=x_{v_1}+x_{v_2}+1$. By the assumption $m\geqslant 92$, we will obtain $\rho>\frac{1+\sqrt{4m-5}}{2}>10$, and then $x_{v_3}\leqslant \frac{2}{\rho-1}\leqslant \frac{2}{9}$. By (\ref{e6}) and (\ref{h1}), we have
\begin{align*}
e(W)\leqslant&\eta_1(U)+\frac{3}{2}\\
=&\sum_{i=1}^3(d_{H}(v_i)-1)x_{v_i}-e(H)+\sum_{H_1\in \mathcal{H}\setminus H}\eta_1(V(H_1))+\frac{3}{2}\\
\leqslant&1+2\cdot\frac{2}{9} -3+\sum_{H_1\in \mathcal{H}\setminus H}\eta_1(V(H_1))+\frac{3}{2}\\
\leqslant&-\frac{1}{18},
\end{align*}
which is a contradiction.
\end{proof}

\begin{claim}\label{one}
$|\mathcal{H}_2|=1$.
\end{claim}
\begin{proof}[\rm{\textbf{Proof of Claim \ref{one}}}]
By (\ref{e5}), (\ref{h1}) and Claim \ref{H3},
\begin{align*}
-\frac{3}{2}\leqslant\eta_1(U)=\sum_{H\in \mathcal{H}_2}\eta_1(V(H))+\sum_{H\in \mathcal{H}_3}\eta_1(V(H))\leqslant-|\mathcal{H}_2|.
\end{align*}
Hence, $|\mathcal{H}_2|\leqslant \frac{3}{2}$ and then $0\leqslant |\mathcal{H}_2|\leqslant 1$.

Now it suffices to show that $\mathcal{H}_2\neq\emptyset$. If $\mathcal{H}_2=\emptyset$, then $G[U\cup\{u^*\}]$ is a star. Assume that $N_G(u^*)=\{v_1,\cdots,v_{d_G(u^*)}\}$. If $W=\emptyset$ and then $G$ is a star. Furthermore, $\rho(G)=\sqrt{m}< \frac{1+\sqrt{4m-5}}{2}$, which is a contradiction. If $W\neq\emptyset$, since $G$ is connected, then there is a vertex $w\in W$ and $w\sim v_i$ for some $i$. Hence we construct a new graph $G'=G-wv_i+u^*v_i$. Since $e(W)\leqslant 1$, we have $G'$ is still gem-free. Since $x_{w}\leqslant x_{u^*}=1$, by Lemma \ref{big}, we have $\rho(G')>\rho(G)$, which is a contradiction. 
\end{proof}

By Claims \ref{H3} and \ref{one}, we have $G[U]\cong K_{1,a}\cup bK_1$ with $a\geqslant 1$. 
\end{proof}

By (\ref{e6}) and  (\ref{h1}), we have 
$$e(W)\leqslant \eta_1(U)+\frac{3}{2}=\sum_{H\in\mathcal{H}_2}\eta_1(H)+\sum_{H\in\mathcal{H}_3}\eta_1(H)+\frac{3}{2}\leqslant \frac{1}{2},$$ which implies that $e(W)=0$. We will further prove that $W=\emptyset$.

\begin{lemma}\label{W}
$W=\emptyset$.
\end{lemma}
\begin{proof}
Suppose to the contrary that there is a vertex $w\in W$. Let $V(bK_1)=\{u_1,\cdots,u_b\}$ and $V(K_{1,a})=\{v,v_1,\cdots,v_a\}$ where $v$ is the central vertex of the star $K_{1,a}$. We claim that $w\nsim v$. Suppose $w\sim v$. We claim that the vertex $w$ is not adjacent to the vertex in $\{v_1,\cdots,v_a\}$. Otherwise, assume that $w\sim v_1$. If $a=1$, then  we construct a new graph $G'=G-wv_1+wu^*$. By the fact $e(W)=0$, we have $G'$ is still gem-free. Since $x_{v_1}\leqslant x_{u^*}=1$, by Lemma \ref{big}, we have $\rho(G')>\rho(G)$, which is a contradiction. If $a\geqslant 2$, then $G[\{v, u^*,v_1,v_2,w\}]$ is the gem, which is a contradiction. Hence $w$ is not adjacent to the vertex in $\{v_1,\cdots,v_a\}$. Then we construct a new graph $G^*=G-wv+wu^*$. Combining the fact $e(W)=0$, we have $G^*$ is still gem-free. Since $x_v\leqslant x_{u^*}=1$, by Lemma \ref{big}, we have $\rho(G^*)>\rho(G)$, which is a contradiction. Hence $w\nsim v$. 

By (\ref{e5}), we have 
$$-\frac{3}{2}\leqslant\eta_1(U)=(a-1)x_v-a-\sum_{j=1}^bx_{u_j}\leqslant -1-\sum_{j=1}^bx_{u_j}.$$
Hence, 
\begin{equation}\label{12}
\sum_{j=1}^bx_{u_j}\leqslant\frac{1}{2}.
\end{equation}
By the eigenquation, we can deduce that $\rho x_{v}=\sum_{i=1}^ax_{v_i}+1$ and $\rho x_{w}\leqslant\sum_{i=1}^ax_{v_i}+\sum_{j=1}^bx_{u_j}\leqslant\sum_{i=1}^ax_{v_i}+\frac{1}{2}$. Hence $x_w<x_v$. If $w\sim u_j$, then let $G'=G-u_jw+u_jv$. Observe that $G'$ is still a gem-free graph with size $m$. Combining the fact $x_w<x_v$ and Lemma \ref{big}, we have $\rho(G')>\rho(G)$, which is a contradiction. Hence $w\nsim u_j$ for $1\leqslant j\leqslant b$. 

Let $W=\{w_1,\cdots,w_c\}$. From above, we know that $N_G(w_i)\subseteq\{v_1,\cdots,v_a\}$, set $d_i=d_G(w_i)$. If $d_i=1$, say $w_i\sim v_1$, then we can construct a new graph $G'=G-w_iv_1+w_iu^*$, which derives a same contradiction. Hence $d_i\geqslant 2$ for $1\leqslant i\leqslant c$. Next, we will modify the structure of $G$ to obtain a contradiction.

\noindent\rule[0pt]{17.5cm}{0.09em}
\textbf{Construction 1.} 

\textbf{Input} Graph $G_0=G$ and $i=1$.

\textbf{Output} Graph $G_c$.

\textbf{Step 1} \textbf{While} $d_i$ even, \textbf{do}

$G_i=G_{i-1}\cup \frac{d_i}{2}K_1$, and $V_i=\{k_1^i,\cdots,k_{\frac{d_i}{2}}^i\}$. 

\textbf{Otherwise} $d_i$ odd, \textbf{do} 

$G_i=G_{i-1}\cup (\frac{d_i-1}{2}+1)K_1$, and  $V_i=\{k_0^i,k_1^i,\cdots,k_{\frac{d_i-1}{2}}^i\}$. 

\textbf{Step 2} \textbf{While} $i<c$, \textbf{do}

$i=i+1$ and back to \textbf{Step 1}.

\textbf{Otherwise,} stop the construction.

\noindent\rule[0pt]{17.5cm}{0.09em}

\textbf{Construction 2.} 

\textbf{Input} Graph $G_0'=G_c$ and $i=1$.

\textbf{Output} Graph $G_c'$.

\textbf{Step 1} \textbf{While} $d_i$ even, \textbf{do}
$$G_i'=G_{i-1}'-\sum_{v_j\in N_G(w_i)}w_iv_j+\sum_{j=1}^{\frac{d_i}{2}}(u^*k_j^i+vk_j^i).$$

\textbf{Otherwise} $d_i$ odd, \textbf{do} 
$$G_i'=G_{i-1}'-\sum_{v_j\in N_G(w_i)}w_iv_j+\sum_{j=1}^{\frac{d_i-1}{2}}(u^*k_j^i+vk_j^i)+u^*k_0^i.$$

\textbf{Step 2} \textbf{While} $i<c$, \textbf{do}

$i=i+1$ and back to \textbf{Step 1}.

\textbf{Otherwise,} stop the construction.

\noindent\rule[0pt]{17.5cm}{0.09em}

By Construction 1, we have $\rho(G)=\rho(G_c)$. Let $G_c^*=G_c'-\{w_1,\cdots,w_c\}$. Then $G_c^*\cong S_{\frac{m+t+3}{2},2}^{-t}$. Hence $G_c^*$ is a connected gem-free graph with size $m$ and $\rho(G_c^*)=\rho(G_c')$.  To obtain a contradiction, we will show $\rho(G_c^*)>\rho(G)$.

Let $\bm{y}=\binom{\bm{x}}{0}$ and $\bm{z}$ be a non-negative eigenvector of $G_c$ and $G_c'$ corresponding to $\rho(G_c)$ and $\rho(G_c')$. Since $k_j^i$ is an isolated vertex in $G_c$ and $w_i$ is an isolated vertex in $G_c'$ for any $1\leqslant i\leqslant c$ and $1\leqslant j\leqslant d_i$, it follows that $y_{k_j^i}=z_{w_i}=0$ for $0\leqslant j\leqslant d_i$ and $1\leqslant i\leqslant c$. Observe that $N_{G_{c'}}(k_j^i)=N_{G_{c'}}(v_{\ell})$ for $j\geqslant 1$ and $1\leqslant\ell\leqslant a$. Hence $z_{k_j^i}=z_{v_{\ell}}$ for $j\geqslant 1$ and $1\leqslant\ell\leqslant a$. Combining the fact $x_{w_i}<x_v\leqslant x_{u^*}=1$, we have $x_{u^*}z_{k_j^i}-x_{w_i}z_{v_{\ell}}=(x_{u^*}-x_{w_i})z_{v_{\ell}}>0$, $x_vz_{k_j^i}-x_{w_i}z_{v_{\ell}}=(x_{v}-x_{w_i})z_{v_{\ell}}>0$.

If $d_i$ is even, let
\begin{align*}\label{e11}
f_{w_i}=&\sum_{j=1}^{\frac{d_i}{2}}(z_{u^*}y_{k_j^i}+y_{u^*}z_{k_j^i}+z_vy_{k_j^i}+y_vz_{k_j^i})-\sum_{j=1}^{d_i}(z_{w_i}y_{v_j}+y_{w_i}z_{v_j})\notag\\
=&\sum_{j=1}^{\frac{d_i}{2}}(y_{u^*}z_{k_j^i}+y_vz_{k_j^i})-\sum_{j=1}^{d_i}y_{w_1}z_{v_j}\notag\\
=&\sum_{j=1}^{\frac{d_i}{2}}(y_{u^*}z_{k_j^i}-y_{w_i}z_{v_j})+\sum_{j=1}^{\frac{d_i}{2}}(y_vz_{k_j^i}-y_{w_i}z_{v_{j+\frac{d_i}{2}}})\notag\\
=&\sum_{j=1}^{\frac{d_i}{2}}(x_{u^*}z_{k_j^i}-x_{w_i}z_{v_j})+\sum_{j=1}^{\frac{d_i}{2}}(x_vz_{k_j^i}-x_{w_i}z_{v_{j+\frac{d_i}{2}}})>0.
\end{align*}
If $d_i$ is odd, let
\begin{align*}
f_{w_i}=& \sum_{j=1}^{\frac{d_i-1}{2}}(z_{u^*}y_{k_j^i}+y_{u^*}z_{k_j^i}+z_vy_{k_j^i}+y_vz_{k_j^i})+(z_{u^*}y_{k_0^1}+y_{u^*}z_{k_0^i})-\sum_{j=1}^{d_i}(z_{w_i}y_{v_j}+y_{w_i}z_{v_j}) \notag\\
=&\sum_{j=1}^{\frac{d_i-1}{2}}(y_{u^*}z_{k_j^i}+y_vz_{k_j^i})+y_{u^*}z_{k_0^i}-\sum_{j=1}^{d_i}y_{w_i}z_{v_j}\notag\\
=&\sum_{j=1}^{\frac{d_i-1}{2}}(y_{u^*}z_{k_j^i}-y_{w_i}z_{v_j})+\sum_{j=1}^{\frac{d_i-1}{2}}(y_vz_{k_j^i}-y_{w_i}z_{v_{j+\frac{d_i}{2}}})+(y_{u^*}z_{k_0^i}-y_{w_i}z_{v_{d_i}})\notag\\
=&\sum_{j=1}^{\frac{d_i-1}{2}}(x_{u^*}z_{k_j^i}-x_{w_i}z_{v_j})+\sum_{j=1}^{\frac{d_i-1}{2}}(x_vz_{k_j^i}-x_{w_i}z_{v_{j+\frac{d_i}{2}}})+(x_{u^*}z_{k_0^i}-x_{w_i}z_{v_{d_i}})>0.
\end{align*}
Then we have
\begin{align*}
(\rho(G_c')-\rho(G_c))\bm{y}^T\bm{z}&=\bm{y}^T(A(G_c')-A(G_c))\bm{z}=\sum_{i=1}^cf_{w_i}>0.
\end{align*}
Hence $\rho(G_c')>\rho(G_c)$. Furthermore, $\rho(G_c^*)=\rho(G_c')>\rho(G_c)=\rho(G)$, which is a contradiction to the choice of $G$.
\end{proof}

\noindent\textbf{Declaration of competing interest:} This paper does not have any conflicts to disclose.

\end{document}